\documentclass[11pt]{amsart}
\usepackage{amsmath,amssymb,amsthm}
\usepackage{mathtools}
\usepackage{booktabs}
\usepackage[hyphens]{url}
\usepackage{hyperref}
\hypersetup{hidelinks}
\usepackage{tikz}

\newtheorem{theorem}{Theorem}
\newtheorem{lemma}[theorem]{Lemma}
\newtheorem{conjecture}[theorem]{Conjecture}
\theoremstyle{remark}
\newtheorem{remark}[theorem]{Remark}

\DeclareMathOperator{\conv}{conv}
\DeclareMathOperator{\area}{area}
\DeclareMathOperator{\width}{width}

\title{The maximum area of the convex hull of a polyhex}
\author{Pragyaan Gaur}
\address{Department of Physics, National Institute of Technology Agartala,
Jirania, Agartala, Tripura 799046, India}
\email{pragyaan\_25dtph013@ug.nita.ac.in}
\thanks{ORCID \href{https://orcid.org/0009-0009-4037-2478}{0009-0009-4037-2478}.}
\date{23 November 2026}

\begin{document}

\begin{abstract}
A polyhex is a connected set of n cells of the regular hexagonal tiling, joined edge to edge,
where each cell has area one. Kurz conjectured an upper bound for the area of the convex hull of
a polyhex, after the corresponding problem for squares was solved. We prove that the convex hull
of a polyhex with n cells has area at most one sixth of the ceiling of $n^2+14n/3$, and that for
every n this value is attained by a path made of three straight runs of nearly equal length. This
proves the conjecture of Kurz and improves his bound by 1/6 when n is divisible by 3. The proof
bounds the hull of the cell centres and three of its widths with a spanning tree of the polyhex.
The upper bound is formalised in the Lean 4 proof assistant with the Mathlib library. We also
show that for every n the extremal polyhex is unique up to rotation, reflection and translation,
which an exhaustive search had confirmed for at most 12 cells. The proof of uniqueness was
generated by Claude at the request of Sascha Kurz, and it is also formalised in Lean.
\end{abstract}

\maketitle

\section{Introduction}

Lay down $n$ hexagonal floor tiles so that they form one connected patch, and then stretch a rubber
band around the whole patch. How much floor can the band enclose? A compact cluster wastes area,
and a single straight line encloses almost nothing beyond the tiles themselves, so the best patch
must strike a balance between the two. For square tiles this question was answered thirty years
ago, and the answer was later extended to every dimension. For hexagons it stayed open, because a
hexagon can push the hull outward in more directions than a square can. In this paper we settle
the hexagonal case. We find the exact maximum for every $n$ and show that it is reached by a
patch shaped like a path with three straight runs of nearly equal length. The proof is short and
elementary, and its main tool, a bound on the area of a convex hull obtained from a spanning tree
of the cells, may be useful for other tiles.

Take $n$ congruent regular polygons in the plane and join them edge to edge so that their union is
connected. We ask for the largest possible area of the convex hull of the union. Bezdek, Bra\ss{}
and Harborth~\cite{BBH} studied this question for segments and for unit squares. For polyominoes
in the plane they found the maximum
$n+\frac12\lfloor\frac{n-1}{2}\rfloor\lfloor\frac n2\rfloor$, and they conjectured a formula in
every dimension. Kurz~\cite{Kurz} proved their conjecture in all dimensions and counted the planar
polyominoes that reach the maximum. At the end of his paper he stated the following conjecture for
hexagons.

\begin{conjecture}[{\cite[Conjecture~2]{Kurz}}]\label{conj:kurz}
The area of the convex hull of any edge-to-edge connected system of $n$ regular unit hexagons is
at most $\frac16\lfloor n^2+\frac{14}{3}n+1\rfloor$.
\end{conjecture}

Here a unit hexagon has area one. The case $n=1$ shows that this is the intended normalisation,
because the bound equals $1$ there.

Steffanov\'a~\cite{Steffanova} studied the same family of questions for polyiamonds and polyhexes.
She solved the case of triangles and described a polyhex shape that she called the washtub. She
reports that she could not prove the upper bound for polyhexes, because a hexagon can enlarge the
hull in more ways than a square can~\cite[Section~3.2.3]{Steffanova}. As far as we know, the case
of hexagons has remained open. Our main result gives the exact maximum.

\begin{theorem}\label{thm:main}
Let $M(n)$ be the largest area of the convex hull of a polyhex with $n$ cells. Then
\[
M(n)=\frac16\left\lceil n^2+\frac{14}{3}n\right\rceil
=\frac16\left(\left\lfloor n^2+\frac{14}{3}n+1\right\rfloor-[\,3\mid n\,]\right).
\]
\end{theorem}

The second form holds because $n^2+\frac{14}{3}n$ is an integer exactly when $3\mid n$. In
particular Conjecture~\ref{conj:kurz} is true. The conjectured bound equals $M(n)$ when $3\nmid n$,
and it is larger than $M(n)$ by $\frac16$ when $3\mid n$. For example $M(3)=\frac{23}{6}$, while
the conjectured bound is $4$. The first values of $6M(n)$ are
\[
6,\ 14,\ 23,\ 35,\ 49,\ 64,\ 82,\ 102,\ 123,\ 147,\ 173,\ 200,
\]
and they form sequence A399934 in the OEIS~\cite{OEIS}. Figure~\ref{fig:washtub} shows a polyhex
with $10$ cells that reaches the maximum. The extremal polyhex is unique.

\begin{theorem}\label{thm:unique}
For every $n\ge1$, a polyhex with $n$ cells whose convex hull has area $M(n)$ is, up to a symmetry
of the lattice and a translation, the path with three straight runs constructed in
Section~\ref{sec:proof}.
\end{theorem}

The proof of Theorem~\ref{thm:unique} in Section~\ref{sec:unique} was generated by the AI model
Claude (Anthropic) when Sascha Kurz asked it to study the cases of equality. It only
uses the cases of equality in the width estimate and in the final arithmetic step.

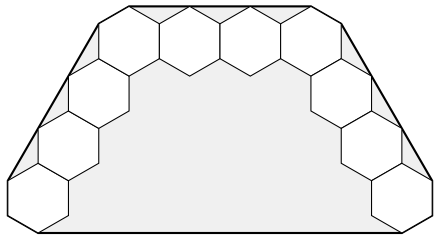
\begin{figure}[ht]
\centering
\begin{tikzpicture}[scale=0.8]
\draw[thick,fill=gray!12] (-0.5000,-0.2887) -- (0.0000,-0.5774) -- (6.0000,-0.5774) -- (6.5000,-0.2887) -- (6.5000,0.2887) -- (6.0000,1.1547) -- (5.0000,2.8868) -- (4.5000,3.1754) -- (1.5000,3.1754) -- (1.0000,2.8868) -- (0.0000,1.1547) -- (-0.5000,0.2887) -- cycle;
\foreach \x/\y in {0.0000/0.0000, 0.5000/0.8660, 1.0000/1.7321, 1.5000/2.5981, 2.5000/2.5981, 3.5000/2.5981, 4.5000/2.5981, 5.0000/1.7321, 5.5000/0.8660, 6.0000/0.0000}
  \draw[fill=white] (\x,\y) ++(30:0.5774) -- ++(150:0.5774) -- ++(210:0.5774) -- ++(270:0.5774) -- ++(330:0.5774) -- ++(30:0.5774) -- cycle;
\end{tikzpicture}
\caption{The figure shows a polyhex with $10$ cells that reaches the maximum. Its convex hull is
shaded and has area $\frac{49}{2}$.}
\label{fig:washtub}
\end{figure}

The proof has three steps. In Section~\ref{sec:setup} we write the hull as a Minkowski sum $P+H$,
where $P$ is the hull of the cell centres and $H$ is one cell. We then express its area through
the area of $P$ and three widths of $P$. In Section~\ref{sec:tree} we bound the area of $P$ and
the widths of $P$ with a spanning tree of the polyhex. The area of $P$ depends on the pairs of tree
edges that are not parallel. The widths depend on the number of tree edges and on the largest
number of parallel tree edges. In Section~\ref{sec:proof} we combine the bounds and give a
polyhex that reaches them. Section~\ref{sec:unique} proves uniqueness. Section~\ref{sec:lean} describes the formal proof, and
Section~\ref{sec:computation} reports the computations. The last section lists open problems.

\section{Coordinates and the area identity}\label{sec:setup}

The centres of the hexagonal tiling form a triangular lattice. We use axial coordinates. In these
coordinates the centres are the points of $\mathbb Z^2$, and the six neighbours of a centre are
obtained by adding $\pm(1,0)$, $\pm(0,1)$ or $\pm(1,-1)$. The linear map
$(q,r)\mapsto\kappa\,(q+\tfrac r2,\tfrac{\sqrt3}{2}r)$ with $\kappa^2=2/\sqrt3$ has determinant
one, and it carries this picture to a regular hexagonal tiling in the Euclidean plane. A linear
map multiplies all areas by the same factor, so we may compute in axial coordinates. There the cell
with centre at the origin is
\[
H=\conv\{\pm(\tfrac13,\tfrac13),\ \pm(-\tfrac13,\tfrac23),\ \pm(\tfrac23,-\tfrac13)\},
\]
and it has area one.

Any edge-to-edge connected system of congruent regular hexagons is part of a single hexagonal
tiling. After a similarity it is therefore $S+H$ for a finite edge-connected set
$S\subset\mathbb Z^2$. Since $H$ is convex,
\[
\conv(S+H)=\conv(S)+H.
\]
From now on we write $P=\conv(S)$.

For a linear functional $f$ and a compact set $K$, let $\width_f(K)=\max_K f-\min_K f$. We use the
three functionals
\[
f_1(q,r)=q+2r,\qquad f_2(q,r)=q-r,\qquad f_3(q,r)=2q+r,
\]
and we write $w_i=\width_{f_i}(P)$. Let $\det(u,v)=u_1v_2-u_2v_1$.

\begin{lemma}\label{lem:segment}
Let $K$ be a compact convex set in the plane and let $g$ be a vector. Then
\[
\area(K+[0,g])=\area(K)+\width_{\det(g,\cdot)}(K).
\]
\end{lemma}

\begin{proof}
The case $g=0$ is clear. Otherwise apply a linear map of determinant one that sends $g$ to
$(\lambda,0)$ with $\lambda>0$. This map turns $\det(g,\cdot)$ into $\lambda$ times the second
coordinate. A horizontal line meets $K+[0,g]$ exactly when it meets $K$. In that case it meets $K$
in a segment, and it meets $K+[0,g]$ in the same segment extended to the right by $\lambda$. We
integrate the extra length $\lambda$ over the second coordinate and obtain the formula.
\end{proof}

\begin{lemma}\label{lem:area}
For every nonempty convex polygon $P$, including a point or a segment,
\[
\area(P+H)=\area(P)+1+\frac{w_1+w_2+w_3}{3}.
\]
\end{lemma}

\begin{proof}
Put $z=(-\frac13,-\frac13)$, $g_1=(\frac23,-\frac13)$, $g_2=(\frac13,\frac13)$ and
$g_3=(-\frac13,\frac23)$. The eight sums $z+\epsilon_1g_1+\epsilon_2g_2+\epsilon_3g_3$ with
$\epsilon_i\in\{0,1\}$ are the six vertices of $H$ and the origin, which is the centre of $H$.
Hence $H=z+[0,g_1]+[0,g_2]+[0,g_3]$. We have $\det(g_1,\cdot)=f_1/3$, $\det(g_2,\cdot)=-f_2/3$ and
$\det(g_3,\cdot)=-f_3/3$. We also have $|\det(g_i,g_j)|=\frac13$ for $i\ne j$. Widths add under
Minkowski sums. We add the three segments one at a time with Lemma~\ref{lem:segment} and get
\[
\area(P+H)=\area(P)+\frac{w_1}{3}+\Bigl(\frac{w_2}{3}+\frac13\Bigr)+\Bigl(\frac{w_3}{3}+\frac23\Bigr).
\qedhere
\]
\end{proof}

\section{Spanning tree estimates}\label{sec:tree}

Let $S$ be a polyhex with $n$ cells and put $s=n-1$. Fix a spanning tree of the adjacency graph of
$S$. Let $a$, $b$ and $c$ be the numbers of its edges parallel to $(1,0)$, $(0,1)$ and $(1,-1)$.
Then $a+b+c=s$. Put $m=\max(a,b,c)$.

\begin{lemma}\label{lem:insert}
Let $K$ be a compact convex set, let $x\in K$ and let $v$ be a vector. Then
\[
\area(\conv(K\cup\{x+v\}))\le\area(K)+\tfrac12\width_{\det(v,\cdot)}(K).
\]
\end{lemma}

\begin{proof}
The case $v=0$ is clear. Otherwise translate so that $x=0$, and apply a linear map of determinant
one that sends $v$ to $(1,0)$. The functional $\det(v,\cdot)$ becomes the second coordinate $y$.
Let $y_+=\max_K y$ and $y_-=\min_K y$. Since $0\in K$, we have $y_-\le0\le y_+$. Define $h$ on
$[y_-,y_+]$ by $h(y)=1-y/y_+$ for $y>0$, by $h(0)=1$, and by $h(y)=1-y/y_-$ for $y<0$.

Every point of $\conv(K\cup\{v\})$ has the form $p=(1-\lambda)k+\lambda v$ with $k\in K$ and
$\lambda\in[0,1]$. The point $q=(1-\lambda)k$ lies in $K$, because $K$ is convex and contains the
origin. The points $p$ and $q$ lie on the same horizontal line, and $p=q+(\lambda,0)$. Let $y=p_2$.
If $y>0$, then $y\le(1-\lambda)y_+$, so $\lambda\le h(y)$. The same bound holds for $y<0$ by the
same argument, and it holds for $y=0$ because $\lambda\le1$.

Fix a horizontal line at height $y$, and let $p$ be a point of $\conv(K\cup\{v\})$ on this line that
is not in $K$. Take $q$ as above. If some point of $K$ on this line were to the right of $p$, then
$p$ would lie between two points of $K$ and so would be in $K$. So $p$ lies to the right of the
segment $K\cap\{y\}$, at distance at most $h(y)$ from its right end. The new part of the line
therefore has length at most $h(y)$. Integrating over $y\in[y_-,y_+]$ gives at most
$\frac12(y_+-y_-)$.
\end{proof}

\begin{lemma}\label{lem:treearea}
The area of $P$ is at most $\frac12(ab+bc+ca)$.
\end{lemma}

\begin{proof}
List the vertices of the tree as $x_0,x_1,\dots,x_s$ so that each $x_j$ with $j\ge1$ has its parent
$x_{p(j)}$ among the earlier vertices. Let $e_j=x_j-x_{p(j)}$ and $K_j=\conv\{x_0,\dots,x_j\}$. For
unit lattice steps, $|\det(e_j,e_i)|$ is $0$ if $e_i$ and $e_j$ are parallel and $1$ if they are
not.

Fix $j\ge1$ and let $N_j$ be the number of indices $1\le i<j$ such that $e_i$ is not parallel to
$e_j$. We claim that the values $\det(e_j,x_i)$ for $i<j$ lie in an interval of length at most
$N_j$. We add the points $x_0,x_1,\dots$ in order. The value at $x_i$ differs from the value at its
parent by $|\det(e_j,e_i)|$, so each new point makes the range longer by at most that amount. By
convexity the same interval contains all values of $\det(e_j,\cdot)$ on $K_{j-1}$.

Lemma~\ref{lem:insert} with $K=K_{j-1}$, $x=x_{p(j)}$ and $v=e_j$ gives
$\area(K_j)\le\area(K_{j-1})+\frac12N_j$. Since $K_0$ is a point, $\area(P)\le\frac12\sum_jN_j$.
Each unordered pair of tree edges that are not parallel is counted exactly once in this sum, so
$\sum_jN_j=ab+bc+ca$.
\end{proof}

\begin{lemma}\label{lem:width}
The widths satisfy $w_1+w_2+w_3\le3s+m$.
\end{lemma}

\begin{proof}
For each $i$ choose cells $N_i$ and $M_i$ of $S$ where $f_i$ takes its smallest and largest
values. Let $Q_i$ be the tree path from $N_i$ to $M_i$. Then $w_i$ is the sum of the increments of
$f_i$ along $Q_i$. We group the three sums by tree edge and get $w_1+w_2+w_3=\sum_ec(e)$. Here
$c(e)$ is the sum of the signed increments of those $f_i$ whose path $Q_i$ uses $e$.

For a unit step, the absolute increments of $(f_1,f_2,f_3)$ are a permutation of $(1,1,2)$, so
$c(e)\le4$. All increments are integers. The equality $c(e)=4$ holds only when $e$ lies on all
three paths and all three signed increments are positive. In every other case $c(e)\le3$. We call
$e$ full when $c(e)=4$.

We show that all full edges are parallel. The edges that the three paths have in common form a
path, and each $Q_i$ runs through this common path in one direction. We orient the common path
along $Q_3$. Then $Q_1$ runs either with it or against it, and the same holds for $Q_2$. These two
choices are the same for every common edge. The steps with a positive $f_3$ increment are $(1,0)$,
$(0,1)$ and $(1,-1)$. Their increments of $(f_1,f_2)$ are $(1,1)$, $(2,-1)$ and $(-1,2)$, and these
three pairs have different sign patterns. A full edge needs the signs of its $f_1$ and $f_2$
increments to match the two fixed choices. Hence at most one direction class contains full edges.
There are at most $m$ full edges, and so $w_1+w_2+w_3\le4m+3(s-m)=3s+m$.
\end{proof}

\begin{remark}
The formal proof replaces the tree by the parent function of Lemma~\ref{lem:treearea}. The tree
path from $N$ to $M$ is described by the chains of ancestors of $N$ and $M$. An edge lies on the
path exactly when it lies in one of the two chains and not in the other. Any two edges in one chain
are comparable, and this fact gives the claim about full edges directly.
\end{remark}

\section{Proof of Theorem~\ref{thm:main}}\label{sec:proof}

\subsection*{Upper bound}
By Lemmas~\ref{lem:area}, \ref{lem:treearea} and~\ref{lem:width},
\[
\area(P+H)\le\frac{ab+bc+ca}{2}+1+s+\frac m3.
\]
Six times the right side is the integer $L=3(ab+bc+ca)+6+6s+2m$. The expression is symmetric in
$a$, $b$ and $c$, so we may assume $c=m$. We write $a=c-x$ and $b=c-y$ with integers $x,y\ge0$. A
direct expansion gives
\[
3(s+1)^2+14(s+1)+2-3L=3(x^2-xy+y^2)-2x-2y+1.
\]
The right side equals $\frac32(x-y)^2+\frac12(x-1)(3x-1)+\frac12(y-1)(3y-1)$. For every integer
$x\ge0$ we have $(x-1)(3x-1)\ge0$, so the right side is nonnegative. Hence $3L\le3n^2+14n+2$. For
the integer $L$ this means $L\le\lceil n^2+\frac{14}{3}n\rceil$, and so
$M(n)\le\frac16\lceil n^2+\frac{14}{3}n\rceil$.

\subsection*{Construction}
Choose integers $a,b,c\ge0$ with $a+b+c=n-1$ such that any two of them differ by at most one and
$b$ is the largest. Start at $(0,0)$ and take $a$ steps $(0,1)$. Then take $b$ steps $(1,0)$, and
then take $c$ steps $(1,-1)$. The $n$ centres are distinct, and they form a polyhex in the shape of
a path with three straight runs. Its centre hull is the quadrilateral with vertices $(0,0)$,
$(0,a)$, $(b,a)$ and $(b+c,a-c)$, and it has area $\frac12(ab+bc+ca)$. Its widths are $w_1=2a+b$,
$w_2=b+2c$ and $w_3=a+2b+c$, so $w_1+w_2+w_3=3s+b$. By Lemma~\ref{lem:area} the hull of this
polyhex has area $\frac12(ab+bc+ca)+1+s+\frac b3$. We compute this value in the three residue
classes of $s$ modulo $3$. For $s=3k$, $3k+1$ and $3k+2$, six times the value is $9k^2+20k+6$,
$9k^2+26k+14$ and $9k^2+32k+23$. In each case this equals
$\lceil n^2+\frac{14}{3}n\rceil$ with $n=s+1$. We believe that this shape is the washtub of
Steffanov\'a~\cite{Steffanova}, and for $n=10$ both hulls have area $\frac{49}{2}$.

\section{Uniqueness of the extremal polyhex}\label{sec:unique}

The argument in this section was generated by Claude at the request of Sascha Kurz, who sent it to
the author on 17~November 2026. Write $d_1=(1,0)$, $d_2=(0,1)$ and
$d_3=(1,-1)$. For integers $i,j,k\ge0$ with $i+j+k=n-1$ let $W(i,j,k)$ be the set of $n$ cells
visited by the walk that starts at $(0,0)$ and takes $i$ steps $d_2$, then $j$ steps $d_1$, then
$k$ steps $d_3$. The polyhex of the construction in Section~\ref{sec:proof} is $W(a,b,c)$ with
$a,b,c$ pairwise differing by at most one and $b$ the largest. We call it the washtub.

We keep the notation of Lemma~\ref{lem:width}. Recall that $c(e)\le4$ for every tree edge $e$, and
that $e$ is full when $c(e)=4$. The increments of $(f_1,f_2,f_3)$ along $d_1$, $d_2$ and $d_3$ are
$(1,1,2)$, $(2,-1,1)$ and $(-1,2,1)$, and they change sign along $-d_1$, $-d_2$ and $-d_3$. Each
step therefore has one functional with increment $\pm2$, which we call its own functional (this is
$f_3$, $f_1$ and $f_2$ for the classes of $d_1$, $d_2$ and $d_3$), and two functionals with
increment $\pm1$. Going through the cases in the proof of Lemma~\ref{lem:width} gives a little more.
If $e$ is not on the path of its own functional, then $c(e)\le2$. If that path traverses $e$ with
increment $-2$, then $c(e)\le0$. Hence $c(e)=3$ exactly when $e$ lies on exactly two of the paths,
one of them the path of its own functional, and both increments are positive.

\begin{lemma}\label{lem:balanced}
Let $S$ be a polyhex with $n$ cells whose convex hull has area $M(n)$. Then for every spanning tree
of $S$ we have $w_1+w_2+w_3=3s+m$, and the multiset $\{a,b,c\}$ is $\{k,k,k\}$ if $n=3k+1$,
$\{k-1,k-1,k\}$ if $n=3k-1$ and $\{k-1,k,k\}$ if $n=3k$.
\end{lemma}

\begin{proof}
In the proof of the upper bound, $6M(n)=6\area(P+H)\le L\le\lceil n^2+\frac{14}{3}n\rceil=6M(n)$,
so both inequalities are equalities. The first one is the sum of the bounds of
Lemmas~\ref{lem:treearea} and~\ref{lem:width}, so each of them is an equality, and in particular
$w_1+w_2+w_3=3s+m$. For the second one put $D=3n^2+14n+2-3L=3(x^2-xy+y^2)-2x-2y+1$. Since
$3n^2+14n\equiv2n\pmod 3$, the equality $L=\lceil n^2+\frac{14}{3}n\rceil$ means $D=2$, $1$ or $0$
when $n\equiv0$, $1$ or $2\pmod 3$. With $g(x)=\frac12(x-1)(3x-1)$ we have
$D=\frac32(x-y)^2+g(x)+g(y)$, and $g(0)=\frac12$, $g(1)=0$ and $g(x)\ge\frac52$ for $x\ge2$. As
$D\le2$, this forces $x,y\in\{0,1\}$. Now $D=1$ at $(0,0)$, $D=0$ at $(1,1)$ and $D=2$ at $(1,0)$
and $(0,1)$, which gives the three multisets.
\end{proof}

\begin{lemma}\label{lem:walk}
Let $S$ be a polyhex with $n\ge2$ cells and let $T$ be a spanning tree of $S$ with
$w_1+w_2+w_3=3s+m$ and $m\ge1$. Then after a symmetry of the lattice, $T$ is the walk
$A^\alpha d_1^m B^\beta$ with $A,B\in\{d_2,d_3\}$ and $\alpha+\beta=s-m$.
\end{lemma}

\begin{proof}
By the proof of Lemma~\ref{lem:width}, the equality means that there are exactly $m$ full edges and
that $c(e)=3$ for every other edge.

\emph{The full edges.} The full edges are parallel, so they form one direction class, and this
class has exactly $m$ edges. An edge that lies on all three paths $Q_1,Q_2,Q_3$ is full, because
otherwise $c(e)\le2$ by the case analysis above. So the set $C$ of edges common to the three paths
is exactly this class. The intersection of three paths in a tree is a path, so $C$ is a path of $m$
parallel edges. A path in a tree does not turn back, so $C$ is a straight run.

\emph{Normalisation.} The lattice symmetry $\rho(q,r)=(-r,q+r)$ satisfies $f_1\circ\rho=f_3$,
$f_2\circ\rho=-f_1$ and $f_3\circ\rho=f_2$, and it maps $d_1\mapsto d_2\mapsto-d_3$ and $d_3\mapsto
d_1$. Every lattice symmetry permutes $\{\pm f_1,\pm f_2,\pm f_3\}$. Changing the sign of a
functional exchanges its smallest and largest cells and reverses the corresponding path, which
leaves every $c(e)$ unchanged. So a lattice symmetry carries the whole setup to a setup of the same
kind, and since $\rho$ permutes the three direction classes cyclically, we may assume that the full
class is the class of $d_1$. All three increments along $d_1$ are positive, so the three paths
traverse $C$ in the direction $d_1$. Write $C$ as $p_0,p_1,\dots,p_m$ with $p_{l+1}=p_l+d_1$.

\emph{The tree is a path.} Remove the edges of $C$ from $T$ and let $T_l$ be the component that
contains $p_l$. A path that contains $C$ can only continue at the two ends of $C$, so each $Q_i$
consists of a part in $T_0$ ending at $p_0$, the run $C$, and a part in $T_m$ starting at $p_m$. An
edge of $T_l$ with $0<l<m$ lies on no path, so $c(e)=0$, which is impossible. Hence
$T_l=\{p_l\}$ for $0<l<m$.

Root $T_0$ at $p_0$. Every $N_i$ lies in $T_0$, and an edge of $T_0$ lies on $Q_i$ exactly when
$N_i$ lies below it. Since $c(e)=3$, every edge of $T_0$ lies on exactly two paths, so it has
exactly two of $N_1,N_2,N_3$ below it. If a vertex of $T_0$ had two child edges, the two sets of
indices $i$ with $N_i$ below them would be disjoint and both of size two, which is impossible with
three indices. So $T_0$ is a path from $p_0$ to an end vertex $w_0$. The sets of indices along this
path are nested and all of size two, so they are equal. If $T_0$ has an edge, two of the $N_i$ are
therefore equal to $w_0$ and the third is $p_0$. The same argument with the cells $M_i$ shows that
$T_m$ is a path from $p_m$ to an end vertex $w_m$.

\emph{The end runs are straight.} Orient $T_0$ from $w_0$ to $p_0$. The same two paths $Q_j$ and
$Q_k$ pass through every edge of $T_0$, and both traverse it in this direction. By the
characterisation of $c(e)=3$, the step along such an edge has positive increments for its own
functional and for one more functional, and these two are the functionals of $Q_j$ and $Q_k$. The
step is not $\pm d_1$, since all edges of that class lie in $C$. The steps $-d_2$ and $-d_3$ have a
negative increment for their own functional. This leaves $d_2$, which needs the pair $\{Q_1,Q_3\}$,
and $d_3$, which needs the pair $\{Q_2,Q_3\}$. The pair is the same for all edges of $T_0$, so all
its steps equal one $A\in\{d_2,d_3\}$. In the same way, with $T_m$ oriented from $p_m$ to $w_m$, all
steps of $T_m$ equal one $B\in\{d_2,d_3\}$. Hence $T=A^\alpha d_1^mB^\beta$, and $\alpha+\beta=s-m$
because the class of $d_1$ has $m$ edges.
\end{proof}

\begin{proof}[Proof of Theorem~\ref{thm:unique}]
For $n=1$ there is nothing to prove, and for $n=2$ any two adjacent cells are congruent. Let
$n\ge3$, let $S$ be extremal and fix a spanning tree. By Lemma~\ref{lem:balanced} we have $m\ge1$,
so Lemma~\ref{lem:walk} shows that after a lattice symmetry and a translation $S$ is the set of
cells of the walk $A^\alpha d_1^mB^\beta$, and the counts of the three classes are the balanced
multiset of Lemma~\ref{lem:balanced}, with $m$ for the class of $d_1$.

For $n=3$ the counts are $\{0,1,1\}$, so the walk has two steps with a turn of sixty degrees, and
all such polyhexes are congruent. For $n\ge4$ all three counts are at least one. So the classes of
$d_2$ and $d_3$ both have edges, which gives $\alpha,\beta\ge1$ and $A\ne B$. If $A=d_2$ then
$S=W(\alpha,m,\beta)$. If $A=d_3$, apply the lattice symmetry $\sigma(q,r)=(q+r,-r)$, which fixes
$d_1$ and exchanges $d_2$ and $d_3$, and get $W(\alpha,m,\beta)$ up to translation. In both cases
$m$ is the largest count and $\alpha$ and $\beta$ are the other two, so $S$ is a washtub.

It remains to see that the two admissible choices of $(i,j,k)$ give congruent polyhexes. The walk
$d_2^id_1^jd_3^k$ traversed backwards is $(-d_3)^k(-d_1)^j(-d_2)^i$. Applying the point reflection
$-\mathrm{id}$ and then $\sigma$ turns it into $d_2^kd_1^jd_3^i$, so $W(i,j,k)$ and $W(k,j,i)$ are
congruent.
\end{proof}

\begin{remark}
The proof does not use the case of equality in Lemma~\ref{lem:treearea}. The note in which Claude produced this proof also reports a check of
Lemma~\ref{lem:walk} by an exhaustive computer search over all pairs of a fixed polyhex and a
spanning tree with $2\le n\le8$, and with a second program for $n\le9$ and all choices of the cells
$N_i$ and $M_i$. It also reports an enumeration of all fixed polyhexes with $n\le11$ and a
branch-and-bound search in exact integer arithmetic for $n\le26$, which agree with
Theorems~\ref{thm:main} and~\ref{thm:unique}. These computations are not used in the proof.
\end{remark}

\section{Formal verification}\label{sec:lean}

We formalised the upper bound in Lean~4 with the Mathlib library. The code is available at
\url{https://github.com/pragyaangaur/Polyhex-Hull}. In the formal statement a cell is a point of
$\mathbb Z\times\mathbb Z$ in axial coordinates. A polyhex is a nonempty finite set of cells whose
induced subgraph in the adjacency graph is connected. The region of a polyhex is the union of the
translates of $H$, and area is Lebesgue measure on $\mathbb R\times\mathbb R$. The theorem
\texttt{kurz\_conjecture\_sharp} states
\[
\operatorname{vol}\bigl(\conv(\text{region }S)\bigr)\le\frac16\left\lceil n^2+\frac{14}{3}n\right\rceil,
\]
and the theorem \texttt{kurz\_conjecture} states the bound of Conjecture~\ref{conj:kurz}. The
development also proves the following facts, which show that the model matches the conjecture.
\begin{itemize}
\item The cell $H$ has area exactly one (\texttt{volume\_hexagon}).
\item The map $(q,r)\mapsto\kappa\,(q+\frac r2,\frac{\sqrt3}{2}r)$ with $\kappa^2=2/\sqrt3$ has
determinant one. The images of two neighbouring vertices of $H$ differ by a rotation of sixty
degrees about the centre, so the image of $H$ is a regular hexagon of area one
(\texttt{euclid\_hexVertex\_succ} and \texttt{volume\_euclid\_hexagon}).
\item The bound holds for the images of the cells in the Euclidean plane
(\texttt{kurz\_conjecture\_euclid}).
\end{itemize}
The development contains no unproved steps. The main theorems use only the three standard axioms
of Lean, which are \texttt{propext}, \texttt{Quot.sound} and \texttt{Classical.choice}.

The formal proof follows Sections~\ref{sec:setup} to~\ref{sec:proof} with three changes.
\begin{itemize}
\item It uses only the inequality in Lemma~\ref{lem:area}.
\item It proves Lemmas~\ref{lem:segment} and~\ref{lem:insert} from one slicing lemma. The slicing
lemma treats stretching along the horizontal axis, and a linear map of determinant one moves it to
other directions.
\item It replaces the spanning tree by an ordering of the cells in which each cell after the first
is adjacent to an earlier cell.
\end{itemize}
The construction in Section~\ref{sec:proof} is not part of the formal proof.

Theorem~\ref{thm:unique} is also formalised. The theorem \texttt{washtub\_unique} states that if
the hull of a polyhex $S$ with $n$ cells has area $\frac16\lceil n^2+\frac{14}{3}n\rceil$, then
$S=\varphi(W(i,j,k))+t$ for some $t\in\mathbb Z^2$ and some $\varphi$ in the group generated by
$\rho$ and $\sigma$, where $i=\lfloor n/3\rfloor$, $j=\lfloor (n+1)/3\rfloor$ and
$k=\lfloor (n-1)/3\rfloor$. It uses the same three axioms. The formal proof differs from
Section~\ref{sec:unique} in one step. Instead of moving the full class to $d_1$ first, it shows
that one of the three widths equals $s+m$, rotates so that this width is $w_3$, and roots a new
spanning tree at a cell where $f_3$ is smallest. In that rooted tree the steps form the walk
directly.

\section{Computations}\label{sec:computation}

We listed all fixed polyhexes with at most $12$ cells with Redelmeier's
algorithm~\cite{Redelmeier}. The program is \texttt{scripts/brute.py} in the repository above. The
numbers of fixed polyhexes agree with sequence A001207 of the OEIS~\cite{OEIS}. For each polyhex we
computed the area of the convex hull of all $6n$ corners of its cells in exact rational
arithmetic. This computation does not use any lemma of this paper. Table~\ref{tab:data} lists the
results. The largest area always equals $M(n)$, and no polyhex has a larger area than
Kurz's bound.

\begin{table}[ht]
\centering
\begin{tabular}{rrrrr}
\toprule
$n$ & fixed polyhexes & $6M(n)$ & maximisers (fixed) & maximisers (free) \\
\midrule
1 & 1 & 6 & 1 & 1 \\
2 & 3 & 14 & 3 & 1 \\
3 & 11 & 23 & 6 & 1 \\
4 & 44 & 35 & 6 & 1 \\
5 & 186 & 49 & 6 & 1 \\
6 & 814 & 64 & 12 & 1 \\
7 & 3652 & 82 & 6 & 1 \\
8 & 16689 & 102 & 6 & 1 \\
9 & 77359 & 123 & 12 & 1 \\
10 & 362671 & 147 & 6 & 1 \\
11 & 1716033 & 173 & 6 & 1 \\
12 & 8182213 & 200 & 12 & 1 \\
\bottomrule
\end{tabular}
\caption{The table lists the results of the search over all polyhexes with at most $12$ cells.
Fixed polyhexes are counted up to translation, and free polyhexes are counted up to the twelve
symmetries of the lattice.}\label{tab:data}
\end{table}

For every $n\le12$ only one free polyhex reaches the maximum, as Theorem~\ref{thm:unique} shows
for all $n$. Polyominoes behave very differently,
because the number of polyominoes that reach the maximum grows like $n^3$~\cite{Kurz}.

\section{Open problems}

Theorem~\ref{thm:unique} settles the question of uniqueness, which an earlier version of this
paper left open as a conjecture. Other natural questions remain.
\begin{itemize}
\item Find a stability version of Theorem~\ref{thm:main} for polyhexes whose hull area is close to
$M(n)$.
\item Determine the set of all hull areas of polyhexes with $n$ cells.
\item Solve the same extremal problem for tiles that are affine images of the regular hexagon.
\end{itemize}

\section*{Acknowledgements and use of AI tools}
The problem is listed as MathDB~380445 on the Principia Math project
(\url{https://principia-math.com}), and the first version of the proof in
Sections~\ref{sec:setup} to~\ref{sec:proof} came from that project. The author thanks
Sascha Kurz, whose paper raised the question, and who obtained the proof of
Theorem~\ref{thm:unique} from Claude, sent it to the author and allowed it to be included here.

AI tools were used in the preparation of this paper. The table below lists each task, together
with the tool and version used for it. Two distinct systems were involved, and they were used for
different things.

\begin{center}
\small
\begin{tabular}{@{}p{0.40\textwidth}p{0.54\textwidth}@{}}
\toprule
Task & AI tool and version \\
\midrule
Choice of the problem
& None. The problem is Conjecture~2 of Kurz~\cite{Kurz}, listed as MathDB~380445 on the Principia
  Math problem list. \\[0.4em]
Generation of the first informal proof of Theorem~\ref{thm:main}, Sections~\ref{sec:setup}
to~\ref{sec:proof}
& The automated proof system of the Principia Math project
  (\url{https://principia-math.com}), used in September 2026. The underlying model and its
  version were not disclosed to the author. \\[0.4em]
\raggedright Proofs of Lemma~\ref{lem:insert} and of the arithmetic step in Section~\ref{sec:proof}, which
replace the corresponding steps of that first proof
& None. These are the author's own. \\[0.4em]
Formalisation of the upper bound in Lean~4 against Mathlib
& Claude Opus~5 (Anthropic, model identifier \texttt{claude-opus-5}), used through Claude Code
  versions 2.1.271 to 2.1.284 between 16~September and 1~October 2026. \\[0.4em]
Writing of the enumeration program \texttt{scripts/brute.py} used in
Section~\ref{sec:computation}
& Claude Opus~5, as above. \\[0.4em]
Writing of the \LaTeX{} source of this paper, including the prose, the tables and the TikZ
figures
& Claude Opus~5, as above. \\[0.4em]
Proof of Theorem~\ref{thm:unique}, Section~\ref{sec:unique}
& Claude (Anthropic), run once by Sascha Kurz, who sent the result to the author on
  17~November 2026. The model version is not known to the author.
  \\[0.4em]
Formalisation of Theorem~\ref{thm:unique} in Lean~4 against Mathlib
& Claude Opus~5.5, used through Claude Code on 21~November 2026. \\[0.4em]
Adaptation of Section~\ref{sec:unique} to the notation of this paper
& Claude Opus~5.5 (Anthropic, model identifier \texttt{claude-opus-5-5}), used through Claude
  Code in November 2026. \\[0.4em]
Checking of the formal proof
& No AI tool. The Lean~4 compiler and Mathlib. \\
\bottomrule
\end{tabular}
\end{center}

No AI tool is an author of this paper. The author read and checked every step of the argument,
and assumes full responsibility for the content and the correctness of this paper, and for the
accuracy of this declaration.

The upper bound does not rest on an informal argument alone. It is formalised in Lean~4 against
Mathlib in the repository cited in Section~\ref{sec:lean}, which compiles with no occurrence of
\texttt{sorry} and depends only on the axioms \texttt{propext}, \texttt{Quot.sound} and
\texttt{Classical.choice}. A reader who wishes to check the main theorem can therefore do so
mechanically, without relying on this declaration.

\end{document}